%% file: dynsys-tensor-evecs.tex
\documentclass[onefignum,onetabnum]{siamart171218}

\input{preamble}
\title{Computing tensor Z-eigenvectors \\ with dynamical systems}

\author{%
  Austin R.\ Benson\thanks{Department of Computer Science, Cornell University (\email{arb@cs.cornell.edu}).}
  \and David F.\ Gleich\thanks{Department of Computer Science, Purdue University (\email{dgleich@purdue.edu}).}
}

\usepackage{amsopn}

\begin{document}

\maketitle

\begin{abstract}
We present a new framework for computing Z-eigenvectors of general tensors
based on numerically integrating a dynamical system that can only converge to a 
Z-eigenvector. Our motivation comes from our recent research on spacey random walks,
where the long-term dynamics of a stochastic process are governed by a dynamical
system that must converge to a Z-eigenvector of a transition probability
tensor. Here, we apply the ideas more broadly to general tensors and find that our 
method can compute Z-eigenvectors that algebraic methods like the higher-order power 
method cannot compute.
\end{abstract}

\section{Preliminaries on Tensor eigenvectors}

Computing matrix eigenvalues is a classic problem in numerical linear algebra and scientific computing.
Given a square matrix $\mA$, the goal is to find a vector-scalar pair $(\vx, \lambda)$ with $\vx \neq 0$ satisfying
\begin{equation}\label{eqn:matrix_evec}
\mA \vx = \lambda \vx.
\end{equation}
The pair $(\vx, \lambda)$ is called the eigenpair, $\vx$ the eigenvector, and $\lambda$ the eigenvalue.
After several decades of research and development, we have, by and large, reliable
methods and software for computing \emph{all} eigenpairs of a given matrix $\mA$. 
(Experts will, of course, be aware of exceptions, but we hope they would agree with the general sentiment of the statement.) 

In numerical \emph{multilinear} algebra, there are analogous eigenvector problems (note the plurality).
For example, given a three-mode cubic tensor $\cmT$ (here meaning that $\cmT$ is a multi-dimensional
$n \times n \times n$ array with entries $\cel{T}_{i,j,k}$, $1 \le i, j, k, \le n$%
\footnote{In true pure mathematical parlance, this is the definition of a hypermatrix, not a tensor
(see Lim~\cite{Lim-2014-tensors} for the precise definitions.)
However, ``tensor'' has become synonymous with a multi-dimensional array of numbers~\cite{Kolda-2009-tensor-decompositions}, and we adopt this terminology here.}), the two most common
tensor eigenvector problems are:
\[ \begin{array}{c@{\qquad\quad}c} 
\text{$Z$-eigenvectors~\cite{Qi-2005-Z-eigenvalues}} & \text{$H$-eigenvectors~\cite{Qi-2005-Z-eigenvalues}} \\
\text{$l^2$-eigenvectors~\cite{Lim-2005-eigenvalues}}& \text{$l^k$-eigenvectors~\cite{Lim-2005-eigenvalues}} \\
\sum_{jk} \cel{T}_{i,j,k} x_j x_k = \lambda x_i,\; 1 \le i \le n & \sum_{jk} \cel{T}_{i,j,k} x_j x_k = \lambda x_i^2,\; 1 \le i \le n \\
\normof[2]{\vx} = 1 & \vx \neq 0 \\
\end{array} \]

We use the ``$Z$''  and ``$H$''  terminology instead of ``$l^2$'' and ``$l^k$''.
Both $Z$- and $H$-eigenvectors are defined for tensors with the dimension equal in all modes
(such a tensor is called \emph{cubic}~\cite{Comon-2008-symmetric}).
The definitions can be derived by showing that the eigenpairs are KKT points
for a generalization of a Rayleigh quotient to
tensors~\cite{Lim-2005-eigenvalues}. 
One key difference between the types is
that $H$-eigenvectors are scale-invariant, 
while $Z$-eigenvectors are not---this is why we put a norm constraint on the vector.
Specifically, if we ignore the norm constraint and scale $\vx$ by a constant,
the corresponding $Z$-eigenvalue would change; 
for $H$-eigenpairs, this is not the case.
If $\cmT$ is symmetric, then it has a finite set of $Z$-eigenvalues and moreover, there
must be a real eigenpair when the order of the tensor (i.e., the number of modes or indices)
is odd~\cite{Cartwright-2013-number}.

This paper presents a new framework for computing $Z$-eigenpairs.
Tensor $Z$-eigenvectors show up in a variety of applications, including
evolutionary biology~\cite{Bini-2011-quadratic,Meini-2011-Perron},
low-rank factorizations and compression~\cite{DeLathauwer-2000-best,Kofidis-2002-best,Anandkumar-2014-tensor-decomp},
signal processing~\cite{DeLathauwer-1997-signal,Kofidis-2001-tensor},
quantum geometry~\cite{Wei-2003-entanglement,Hu-2016-entanglement},
medical imaging~\cite{Qi-2008-kurtosis}, and
data mining~\cite{Benson-2015-tensor,Gleich-2015-mlpr,Wu-2016-tensor-co-clustering,Benson-2018-hypergraphs}.
All real eigenpairs can be computed with a Lassere type semidefinite 
programming hierarchy~\cite{Cui-2014-real,Nie-2014-sdp,Nie-2017-eigenvalues}, 
but the problem of computing them remains NP-hard~\cite{Hillar-2013-NP-hard} and
the scalability of such methods is limited (see \cref{sec:scalability} for experiments).
Thus, we still lack robust and scalable general-purpose methods for
computing these eigenvectors. 

We introduce two special cases of tensor contractions that will be useful:
\begin{enumerate}
\item The \emph{tensor apply} takes a cubic tensor and a vector and produces a vector, akin to Qi's notation \cite{Qi-2005-Z-eigenvalues}:
\[
\begin{array}{@{\quad} l l l}
\text{three-mode tensor} & \vy = \cmT \vx^2 & y_i =  \sum_{j,k}\cel{T}_{i,j,k} x_j x_k  \\
\text{$m$-mode tensor} & \vy = \cmT \vx^{m-1} & y_i =  \sum_{i_2, \ldots, i_n}\cel{T}_{i_1, \ldots, i_m} x_{i_2} \cdots x_{i_m}
\end{array}
\]
\item The \emph{tensor collapse} takes a cubic tensor and a vector and produces a matrix:
\[
\begin{array}{@{\quad} l l l}
\text{three-mode tensor} & \mY = \cmT[\vx] & \mY = \sum_{k} \cel{T}_{:,:,k}x_k \\
& & Y_{ij} = \sum_{k} \cel{T}_{i,j,k}x_k \\
\text{$m$-mode tensor} & \mY = \cmT[\vx]^{m-2} & \mY = \sum_{i_3,\ldots,i_m} \cel{T}_{:,:,i_3, \ldots i_m}x_{i_3} \cdots x_{i_m} \\
& & Y_{ij} = \sum_{i_3,\ldots,i_m} \cel{T}_{i,j,i_3, \ldots, i_m}x_{i_3} \cdots x_{i_m}.
\end{array}
\]
\end{enumerate}
For the tensor collapse operator, the ``:'' symbol signifies taking all entries along that index, 
so $\cel{T}_{:,:, k}$ is a square matrix.
The tensor may not be symmetric, but we are always contracting onto
the first mode (tensor apply) or first and second modes (tensor collapse);
we assume that $\cmT$ has been permuted in the appropriate manner for the problem at hand.
With this notation, the $Z$-eigenvector problem can be written as
\begin{equation}
\cmT \vx^{m-1} = \lambda \vx,\;\;\; \normof[2]{\vx} = 1.
\end{equation}

The crux of our computational method is based on the following observation
that relates tensor and matrix eigenvectors.
\begin{observation}\label{obs:main}
A \emph{tensor $Z$-eigenvector} $\vx$ of an $m$-mode tensor must be a
\emph{matrix eigenvector} of the collapsed matrix $\cmT[\vx]^{m-2}$, i.e.,
\begin{equation}
\cmT \vx^{m-1} = \lambda \vx \iff \cmT[\vx]^{m-2} \vx = \lambda \vx.
\end{equation}
\end{observation}
The catch, of course, is that the matrix itself depends on the tensor
eigenvector we want to compute, which we do not know beforehand. Therefore, we still have a nonlinear problem.
Although we have used this dynamical system for computing eigenvectors in
prior work~\cite{Benson-2017-srw}, this is the first time that the connection
with the matrix eigenproblem has been used for algorithm design.

\section{A dynamical systems framework for computing $Z$-eigenvectors}

\Cref{obs:main} provides a new perspective on the tensor $Z$-eigenvector problem.
Specifically, tensor $Z$-eigenvectors are matrix eigenvectors,
just for some unknown matrix.
Our computational approach is based on the following continuous-time dynamical system,
for reasons that we will make clear in \Cref{sec:srw}:
\begin{equation}\label{eqn:dynsys}
\frac{d\vx}{dt} = \Lambda(\cmT[\vx]^{m-2}) - \vx.
\end{equation}
Here, $\Lambda$ is some fixed map that takes as input a matrix and 
produces as output some prescribed eigenvector of the matrix with unit norm.
For example, on an input $\mM$, $\Lambda$ could be defined to compute several objects:
\begin{enumerate}
\item the eigenvector of $\mM$ with $k$th smallest/largest magnitude eigenvalue
\item the eigenvector of $\mM$ with $k$th smallest/largest algebraic eigenvalue
\item the eigenvector of $\mM$ closest in distance to a given vector $\vv$.
\end{enumerate}
We resolve the ambiguity in the sign of the eigenvector by picking 
the sign based on the first element. In the case of multiple eigenvectors sharing 
an eigenvalue, we propose using the closest eigenvector to $\vx$, although we
have not evaluated this technique.

\begin{figure}[tb]
\centering
\lstset{%
language         = Julia,
basicstyle       = \footnotesize \ttfamily,
keywordstyle     = \bfseries\color{blue},
stringstyle      = \color{magenta},                                                                 
commentstyle     = \color{teal},
showstringspaces = false,
numbers=left
}
\scalebox{0.9}{
\begin{lstlisting}
using LinearAlgebra

function tensor_apply(T::Array{Float64,3}, x::Vector{Float64})
    n = length(x)
    y = zeros(Float64, n)
    for k in 1:n; y += T[:, :, k] * x * x[k]; end
    return y
end

function tensor_collapse(T::Array{Float64,3}, x::Vector{Float64})
    n = length(x)
    Y = zeros(Float64, n, n)
    for k in 1:n; Y += T[:, :, k] * x[k]; end
    return Y
end

function dynsys_forw_eul(T::Array{Float64,3}, h::Float64, niter::Int64)
    function dx_dt(u::Vector{Float64})  # Derivative
        F = eigen(tensor_collapse(T, u))
        ind = sortperm(abs.(real(F.values)))[1]
        v = F.vectors[:, ind]
        return sign(v[1]) * v - u  # sign consistency
    end
    x = normalize(ones(Float64, size(T, 1)), 1)  # starting point
    eval_hist = [x' * tensor_apply(T, x)]
    for _ = 1:niter
        x += h * dx_dt(x)  # forward Euler
        push!(eval_hist, x' * tensor_apply(T, x))  # Rayleigh quotient
    end
    return (x, eval_hist)  # guess at evec and history of evals
end
\end{lstlisting}}
\vspace{-0.25cm}
\caption{Julia implementation of the dynamical system for a
    3-mode tensor with a map $\Lambda$ that picks
    the largest magnitude real eigenvalue and numerical integration with the forward Euler method.
    Code snippet is available at \url{https://gist.github.com/arbenson/f28d1b2de9aa72882735e1be24d05a7f}.
    A more expansive code is available at \url{https://github.com/arbenson/TZE-dynsys}.}
\label{fig:julia_FE}
\end{figure}

\begin{proposition}\label{prop:dynsys_conv}
Let $\Lambda$ be a prescribed map from a matrix to one of its eigenvectors.
Then if the dynamical system in \cref{eqn:dynsys} converges to a non-zero
solution, it must converge to a tensor $Z$-eigenvector.
\end{proposition}
\begin{proof}
If the dynamical system converges, then it converges to a stationary point.
Any stationary point has zero derivative, so
\begin{align*}
\frac{d\vx}{dt} = 0 
 \iff \Lambda(\cmT[\vx]^{m-2}) = \vx 
& \iff \cmT[\vx]^{m-2}\vx = \lambda\vx \text{ for some $\lambda$ that depends on $\Lambda$} \\
& \iff \cmT\vx^{m-1} = \lambda \vx.
\end{align*}
\end{proof}

One must be a bit careful with the input and output values of $\Lambda$.
If $\cmT$ is not symmetric, then $\cmT[\vx]^{m-2}$ might not be diagonalizable, and
we may have to deal with complex eigenvalues. To keep the dynamical system real-valued,
one could always modify the map $\Lambda$ to output the real part.
However, the tensor need not be symmetric (nor $\cmT[\vx]^{m-2}$ normal for all $\vx$) for
the dynamical system to maintain real values.
In fact, our motivation for this dynamical system comes from a tensor that is not necessarily symmetric,
which we will discuss in \Cref{sec:srw}.

\Cref{prop:dynsys_conv} leads to a broad framework for computing $Z$-eigenvectors:
\begin{enumerate}
\item Choose a map $\Lambda$ and a numerical integration scheme.
\item Numerically integrate \cref{eqn:dynsys}.
\end{enumerate}
Different choices of $\Lambda$ may converge to different $Z$-eigenvectors
and different numerical integration schemes may lead to different convergence properties.
\Cref{fig:julia_FE} shows a concrete example, where $\Lambda$ picks the eigenvector
corresponding to eigenvalue with largest magnitude real part, along
with the forward Euler numerical integration scheme.

The dynamical system in \cref{eqn:dynsys} has no dependence on the time $t$.
Thus, the system might be a good candidate for an explicit solution; however,
this would require integrating the map $\Lambda$, for which an explicit solution 
is unclear in general. Thus, we focus on numerical integration.

\subsection{Forward Euler and diagonal tensors}

As an illustrative example, we consider
the special case of using the forward Euler
numerical integration scheme for computing the tensor eigenvalues of an $n$-dimensional,
$m$-mode diagonal tensor $\cmT$.
Without loss of generality, assume that the diagonal entries of $\cmT$ are decreasing in order 
so that  $\cT_{i, \ldots, i} < \cT_{j, \ldots, j}$ if $i > j$.
This tensor has at least $n$ $Z$-eigenpairs:\ $(\ve_i, \cT_{i, \ldots, i})$ for $1 \le i \le n$,
where $\ve_i$ is the $i$th standard basis vector.
Suppose that we want to compute the $i$th eigenvector and set $\Lambda$ to 
select the unit-norm eigenvector closest to $\ve_i$ in angle.
Since $\cmT[\vx]^{m-2}$ is diagonal, its eigenvectors are the standard basis vectors, and
$\Lambda(\cmT[\vx]^{m-2}) = \ve_i$.
Let $\vr_{k} = \vx_{k} - \ve_i$ be the residual at the $k$th iteration.
If the step size is $h$, then
\begin{align*}
\| \vr_{k+1} \| = \| \vx_{k+1} - \ve_i\| &= \| \vx_k + h(\ve_i - \vx_k) - \ve_i\| \\
&= (1 - h)\| \vx_k - \ve_i \| = (1 - h)\| \vr_{k} \| = (1 - h)^k\| \vr_{0} \|.
\end{align*}
Thus, the forward Euler scheme converges if $h \le 1$ and converges in one step if $h = 1$.
\Cref{fig:dynamics} (left) illustrates the dynamics for an example tensor $\cmT$ ($n = 3$, $m = 3$)
with $\cT_{1,1,1} = 5$, $\cT_{2,2,2} = 2$, and $\cT_{1,1,1} = 1$. 
In this case, the entire surface of the three-dimensional sphere
is a basin of attraction for this eigenvector, which is consistent with our convergence analysis.
In fact, for this specific case, we have the closed form solution 
\begin{align*}
 \vx(t) = e^{-t} [\vx(0) - \ve_i] + \ve_i,
\end{align*}
and we have exponential convergence to a solution, consistent with \Cref{fig:dynamics} (left).

\begin{figure}[tb]
\includegraphics[width=0.495\textwidth]{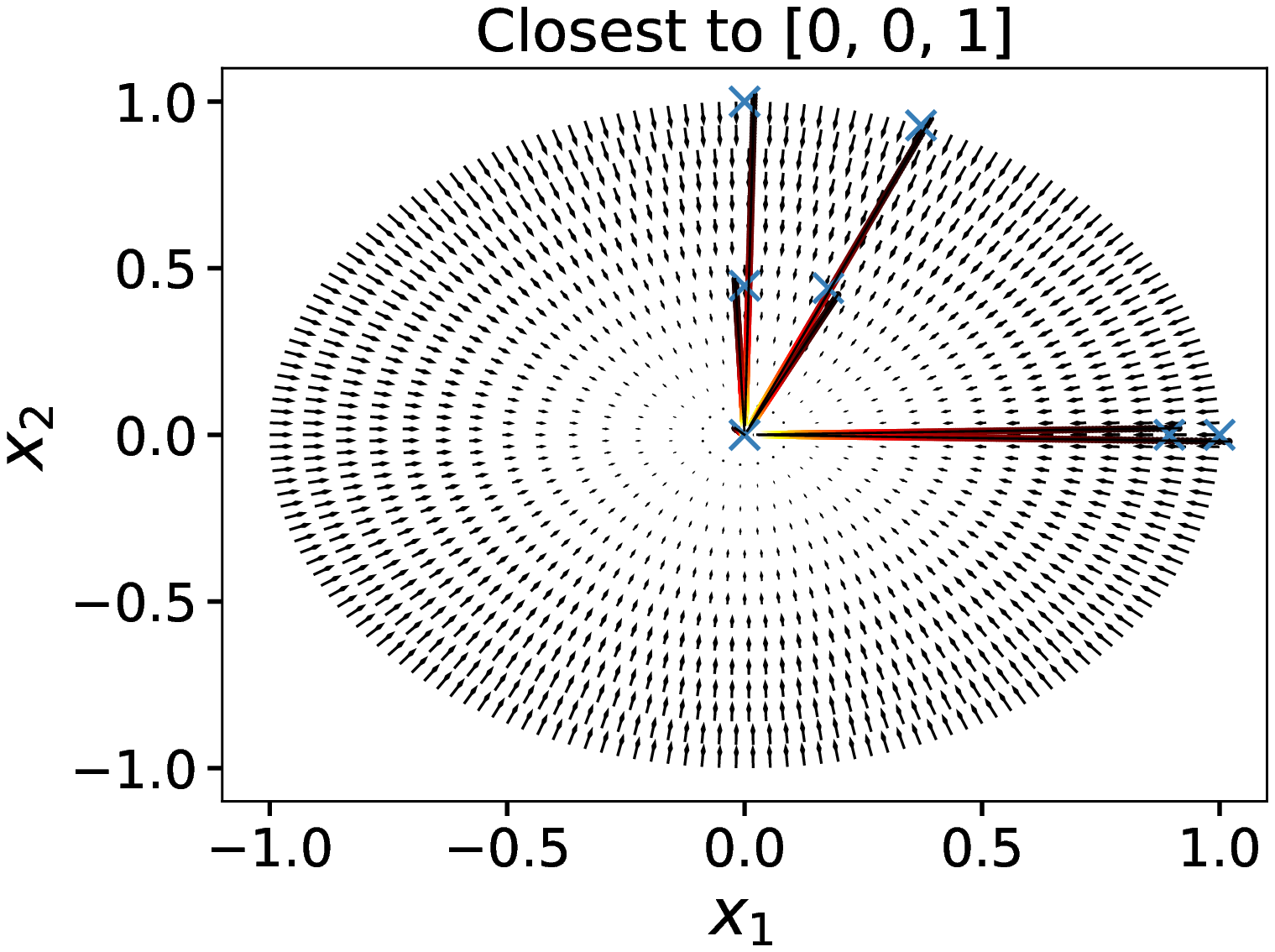}
\includegraphics[width=0.495\textwidth]{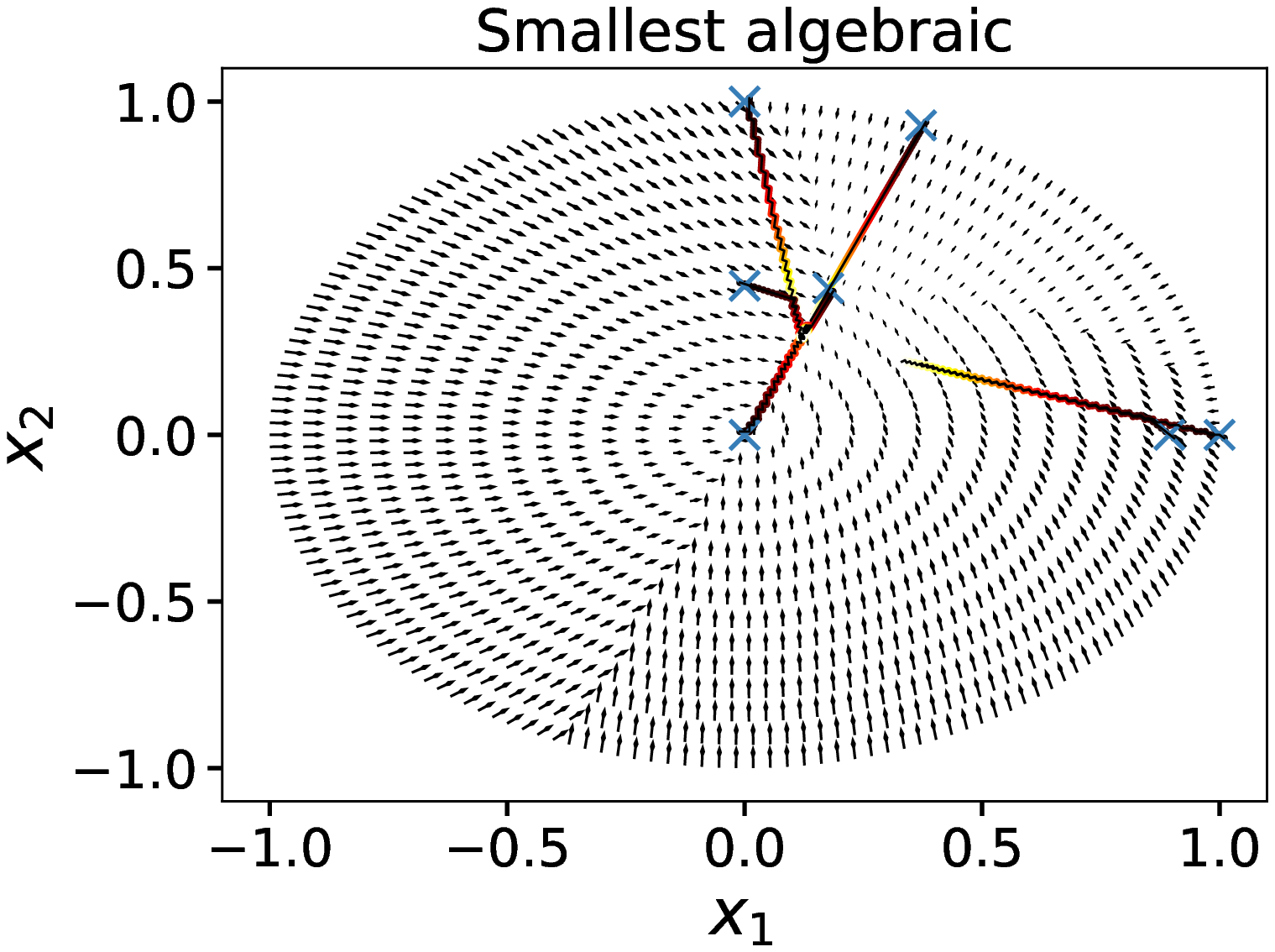}
\vspace{-9mm}
\caption{%
Vector field of dynamical systems from \cref{eqn:dynsys} of points on the
surface of the three-dimensional unit sphere for a three-dimensional,
third-order diagonal tensor with entries $5$, $2$, and $1$ along the diagonal.
Fixed points are marked with a blue `X'.  Sample trajectories of the dynamical
system are illustrated with darker red points corresponding to earlier time and
lighter yellow corresponding to later time (computed with forward Euler with step size 0.01).
\textbf{(Left)} When the map $\Lambda$ is the eigenvector closest in angle to
the third standard basis vector, the entire surface is a basin of attraction for
this vector, which is a $Z$-eigenvector of the tensor. The trajectories head towards
this eigenvector.
\textbf{(Right)} When the map $\Lambda$ is the eigenvector corresponding to the
smallest algebraic eigenvalue, the $Z$-eigenvector corresponding to third standard basis
vector has no basin of attraction. The map is undefined at the
attracting eigenvector where $(x_1, x_2) \approx (0.18, 0.44)$ because 
the eigenspace corresponding to the smallest algebraic eigenvalue has dimension greater than one.
}
\label{fig:dynamics}
\end{figure}

However, our analysis relies on the particular choice of $\Lambda$.
Suppose instead that we choose $\Lambda$ to select the eigenvector corresponding to 
the smallest algebraic eigenvalue, and we are trying to compute the eigenvector $\ve_3$
of a $3 \times 3 \times 3$ diagonal tensor $\cmT$ with strictly decreasing diagonal entries.
Moreover, suppose we have a starting iterate 
$\vx_0 = \bmat{\eps/2 & \eps/2 & 1 - \eps}^T$,
which is close to the $Z$-eigenvector $\ve_3$.
Then
\begin{align*}
\Lambda(\cmT[\vx_0]) = \Lambda\left(\bmat{
\frac{\eps}{2} \cT_{1,1,1} & 0 & 0 \\
0 & \frac{\eps}{2} \cT_{2,2,2} & 0 \\
0 & 0 & (1 - \eps) \cT_{3,3,3}
}\right) = \ve_2,
\end{align*}
if $\eps$ is sufficiently small. Forward Euler integration with step size $h$ gives the next iterate
$\vx_1 = \vx_0 + h(\Lambda(\cmT[\vx_0]) - \vx_0) 
= \bmat{(1 - h)\eps & (1 - h)\eps + h & 1 - \eps - h}^T$,
which is further away from the the $Z$-eigenvector $\ve_3$ than $\vx_0$.
Thus, there is no basin of attraction for the $Z$-eigenvector $\ve_3$
with this particular choice of map $\Lambda$.

It turns out that this is a case where the dynamical system does not converge.
The system is ill-defined for some points and moreover, these points are attractors
(specifically, the eigenvector $(x_1, x_2, x_3) \approx (0.18, 0.44, 0.88)$ in the example above; see \Cref{fig:dynamics}, right).
In general, for some time, the dynamical system will evolve in the direction of $\ve_i$, where $\cT_{i,i,i}x_i < \min_{j}\cT_{j,j,j}x_j$
for $i \neq j$, $i,j \in \{1,2,3\}$. Along this direction, the $i$th coordinate of the vector
increases until $\cT_{i,i,i}x_i = \cT_{j,j,j}x_j$ for some $j \neq i$.
At this point, the map is ill-defined, since the eigenspace corresponding to the
smallest eigenvalue of $\cmT[\vx]$ has dimension at least two.
Since the diagonal tensor entries are distinct by assumption, this is not a fixed point.
We can disambiguate the map at these ambiguous points.
However, any way of doing so besides artificially mapping the vector to a $Z$-eigenvector of $\cmT$, would result in immediate
attraction back to one of these ambiguous points.

\subsection{Spacey random walks motivation for the dynamical system}\label{sec:srw}

The motivation for the dynamical system comes from our previous analysis of a
stochastic process known as the ``spacey random walk'' that relates tensor
eigenvectors of a particular class of tensors to a stochastic
process~\cite{Benson-2017-srw}.  Specifically, the class of tensors are
irreducible \emph{transition probability tensors} (any irreducible tensor $\cmP$
with $\sum_{i_1=1}^{n}\cP_{i_1,i_2,\ldots,i_m} = 1$ for $1 \le i_2, \ldots, i_m \le n$).
For simplicity, we discuss a three-mode transition probability tensor
$\cmP$, where the entries can be interpreted as coming from a second-order
Markov chain---the entry $\cP_{i,j,k}$ is the probability of transitioning to
state $i$ given that the last two states were $j$ and $k$. Due to the theory of
Li and Ng~\cite{Li-2013-tensor-markov-chain}, there exists a tensor $Z$-eigenvector $\vx$
with eigenvalue $1$ satisfying
\begin{equation}\label{eqn:spacey_evec}                                                                                                                                       
\textstyle \cmP\vx^2 = \vx, \quad \sum_{i=1}^{n} x_{i} = 1, \quad x_{i} \ge 0.                                                                                           
\end{equation}

The vector $\vx$ is stochastic, but it does \emph{not} represent the                                                                                                             
stationary distribution of a Markov chain. Instead, we showed that $\vx$ is the                                                                                               
limiting distribution of a non-Markovian, generalized vertex-reinforced                                                                                                     
random walk~\cite{Benaim-1997-vrrw} that we called the 
spacey random walk~\cite{Benson-2017-srw}. 
In the $n$th step of a spacey random walk,                                                                                                       
after the process has visited states $X_1, \ldots, X_{n}$, 
it \emph{spaces out} and forgets its second last state (that is, the                                                                                    
state $X_{n-1}$). It then invents a new history state $Y_{n}$ by randomly                                                                                                     
drawing a past state $X_{1}, \ldots, X_{n}$. Finally, it transitions to                                                                                                       
$X_{n+1}$ via the second-order Markov chain represented by $\cmP$ as if its last                                                                                               
two states were $X_{n}$ and $Y_{n}$, i.e., it transitions to $X_{n+1}$ with                                                                                                   
probability $\cP_{X_{n+1}, X_n, Y_n}$. (In contrast, a true second-order Markov chain                                                                                         
would transition with probability $\cP_{X_{n+1}, X_n, X_{n-1}}$.)

Using results from Bena\"{i}m~\cite{Benaim-1997-vrrw}, we showed that 
the long-term dynamics of the spacey random walk for an $m$-mode 
transition probability tensor are governed by the following dynamical 
system~\cite{Benson-2017-srw}:
\begin{equation}\label{eqn:srw_dynsys}
\frac{d\vx}{dt} = \Pi(\cmP[\vx]^{m-2}) - \vx,
\end{equation}
where $\Pi$ is a map that takes a column-stochastic transition matrix and maps it to 
the Perron vector of the matrix.
In other words, if the spacey random walk converges, it must converge to an attractor
of the dynamical system in \cref{eqn:srw_dynsys}.
The dynamical system in \cref{eqn:srw_dynsys} is a special case of the more general
system in \cref{eqn:dynsys}, where the map $\Lambda$ picks
the eigenvector with largest algebraic eigenvalue (the Perron vector), and the tensor
has certain structural properties (it is an irreducible transition probability tensor).

To summarize, our prior work studied a specific case of the general 
dynamical system in \cref{eqn:dynsys} to understand the stochastic 
process behind principal $Z$-eigenvectors of transition probability tensors.
The general dynamical system provides a new framework for computing
general tensor eigenvectors---\emph{if} the dynamical system in \cref{eqn:dynsys}
converges, then it converges to a tensor $Z$-eigenvector.
The dynamical system may not have an attractor~\cite{Peterson-2018-comms}, 
but it usually does in practice (see \cref{sec:numerical}).

\subsection{Relationship to the Perron iteration}
Bini, Meini, and Poloni derived a Perron iteration to compute
the minimal nonnegative solution of the equation
\begin{equation}\label{eqn:QVE}
\vx = \va + \cmB\vx^2,
\end{equation}
where $\va$ and $\cmB$ are nonnegative
and the all-ones vector $\ve$ is a (non-minimal) nonnegative 
solution~\cite{Bini-2011-quadratic,Meini-2011-Perron}.
The Perron iteration for computing the minimal nonnegative solution is
\begin{equation}\label{eqn:PI}
\vx_{k+1} = \Pi[\mF + \cmB[\ve] - \cmB[\vx_k]],
\end{equation}
where $\mF = \sum_{j} \cmB_{:,j,:}$, $\ve$ is the vector of all ones,
and $\Pi$ maps a nonnegative matrix to its Perron vector with unit $1$-norm.
Suppose that $\vx_{0} \ge 0$ and $\ve^T\vx_{0} = 1$.
Then every iterate $\vx_k$ is stochastic and 
\begin{equation}
\mF = \cmW[\vx_k],\; \cW_{i,j,\ell} = F_{i,j},\qquad \cmB[\ve] = \cmZ[\vx_k],\; \cZ_{i,j,\ell} = [\cmB[\ve]]_{i,j}.
\end{equation}
Thus, we can re-write the Perron iteration in \cref{eqn:PI} as
$\vx_{k+1} = \Pi(\cmT[\vx_k])$, where $\cmT = \cmW + \cmZ - \cmB$.
These iterates are equivalent to forward Euler integration of the dynamical
system in \cref{eqn:dynsys} with unit step size and eigenvector map $\Lambda = \Pi$.

Meini and Poloni~\cite{Meini-2017-Perron} derived a similar Perron iteration for the solution to \cref{eqn:spacey_evec}
for the case of a $3$-mode transition probability tensor.
The algorithm first computes a minimal sub-stochastic nonnegative vector $\vm$
satisfying $\vm = \cmP\vm^{2}$ using a Newton method.
The Perron iteration for the transition probability tensor is then
$\vx_{k+1} = \Pi(\cmT[\vx_k])$ for $\cT_{i,j,\ell} = \cP_{i,j,\ell}(1 + m_j + m_{\ell})$.
The iterates are again equivalent
to forward Euler integration of \cref{eqn:dynsys} with unit step size.

\subsection{Relationship to the shifted higher-order power method}

The shifted higher-order power method~\cite{Kolda-2011-sshopm}
can be derived by noticing that
\begin{equation}
(1 + \gamma)\lambda \vx = \cmT \vx^{m-1} + \gamma\lambda \vx
\end{equation}
for any eigenpair. This yields the iteration
\begin{equation}\label{eqn:shifted_hopm}
\vx_{k+1} = \frac{\frac{1}{ 1 + \gamma}\left(\cmT \vx_k^{m-1} + \gamma\vx_k\right)}{\| \frac{1}{ 1 + \gamma}\left(\cmT \vx_k^{m-1} + \gamma\vx_k\right) \|_2}
\end{equation}
for any shift parameter $\gamma$
(the case where $\gamma = 0$ is just the classical
``higher-order power method''~\cite{DeLathauwer-2000-best,Regalia-2000-hopm,Kofidis-2002-best}).
Kolda and Mayo showed that when $\cmT$ is symmetric, 
the iterates in \cref{eqn:shifted_hopm} converge monotonically
to a tensor eigenvector given an appropriate shift $\gamma$.

If $\cmT = \cmP$ for some transition probability tensor $\cmP$
and we are interested in the case when $\lambda = 1$ and we 
normalize via $\normof[1]{\vx} = 1$,
then one can also derive these iterates by the dynamical system
\begin{equation}\label{eqn:shifted_hopm_tpt_dynsys}
\frac{d\vx}{dt} = \cmP \vx^{m-1} - \vx
\end{equation}
(c.f.~\cref{eqn:srw_dynsys}).
If this dynamical system converges ($d\vx/dt = 0$), then
$\vx = \cmP \vx^{m-1}$,
and $\vx$ is a tensor $Z$-eigenvector with eigenvalue $1$.
If we numerically integrate \cref{eqn:shifted_hopm_tpt_dynsys} using the forward Euler method 
with step size $h = 1 / (1 + \gamma)$ 
and any starting vector $\vx_0$ satisfying $\vx_0 \ge 0$ and $\| \vx_0 \|_1 = 1$, 
then the iterates are
\begin{align}
\vx_{k+1} &= \vx_k + \frac{1}{1 + \gamma}\left(\cmP\vx_k^{m-1} - \vx_k\right) \\
&= \frac{1}{1 + \gamma}\left(\cmP\vx_k^{m-1} + \gamma\vx_k\right) 
= \frac{\frac{1}{1 + \gamma}\left(\cmP\vx_k^{m-1} + \gamma\vx_k\right)}{\| \frac{1}{1 + \gamma}\left(\cmP\vx_k^{m-1} + \gamma\vx_k\right)\|_1},
\end{align}
which are the same as the shifted higher-order power method iterates in \cref{eqn:shifted_hopm}.
The last equality follows from the fact that $\| \vx_k \|_1 = 1$ and $\vx_k \ge 0$,
which is true by a simple induction argument:
the base case holds by the initial conditions and
\begin{equation}
\| \cmP\vx_k^{m-1} + \gamma\vx_k \|_1 = \| \cmP\vx_k^{m-1} \|_1 + \gamma = 1 + \gamma
\end{equation}
since $\cmP\vx_k^{m-1}$ and $\vx_k$ are both stochastic vectors.

With a general tensor $\cmT$, we can either enforce normalization by evolving the dynamical system
over, say, a unit sphere, or we can let the vector $\vx$ be unnormalized. 
The latter case gives a more direct connection to SS-HOPM.
In this case, any vector $\vx$ where $\cmT \vx^{m-1} = \normof[2]{\vx} \vx$ is
a tensor $Z$-eigenvector. This leads to the following dynamical system:
\begin{equation}\label{eqn:shifted_hopm_dynsys}
\frac{d\vx}{dt} = \cmT \vx^{m-1} - \normof[2]{\vx}\vx.
\end{equation}
If $d\vx / dt = 0$, then $\normof[2]{\vx}\vx = \cmT \vx^{m-1}$,
so $\vx$ is a $Z$-eigenvector of $\cmT$ with eigenvalue $\normof[2]{\vx}$.
Now suppose that we numerically integrate the dynamical system in \cref{eqn:shifted_hopm_dynsys} by
\begin{enumerate}
\item taking a forward Euler step to produce the iterate $\vx_{k+1}'$; and
\item projecting $\vx_{k+1}'$ onto the unit sphere by $\vx_{k + 1} = \vx_{k+1}' / \| \vx_{k+1}' \|_2$.
\end{enumerate}
If the step size of the forward Euler method is $h = 1 / (1 + \gamma)$, then
\begin{align}
\vx'_{k + 1} &= \vx_k + \frac{1}{1 + \gamma}(\cmT \vx_k^{m-1} - \normof[2]{\vx_k} \vx_k)  
= \frac{1}{1 + \gamma}\left(\cmT\vx_k^{m-1} + \gamma\vx_k\right) 
\end{align}
since $\normof[2]{\vx_k} = 1$. The projection onto the unit sphere then gives the shifted
higher-order power method iterates in \cref{eqn:shifted_hopm}.

\section{Numerical examples}\label{sec:numerical}

We now show that our method works on two test tensors used in prior work.
\Cref{sec:ex_KM} shows that our approach can compute all eigenvalues
of a specific tensor, while the (shifted) higher-order power method cannot
compute all of the eigenvalues.
\Cref{sec:ex_CDN} verifies that our approach can compute all eigenvalues
of a tensor whose eigenvalues were found with semi-definite programming (SDP).
Finally, \cref{sec:scalability} shows that our
method is faster than SS-HOPM and the SDP method.

\newsavebox{\cmTbox}
\savebox{\cmTbox}{$\cmT \vx^m$}
\newsavebox{\cmTrq}
\savebox{\cmTrq}{$\vx_k^T \cmT \vx_k^{m-1}$}

\begin{figure}[t]
\setlength{\tabcolsep}{6pt}
\centering
\begin{tabular}{@{}llccccccc@{}}
  \toprule
  $\lambda$ & Type & S-HOPM & SS-HOPM & V1 & V2 & V3 & V4 & V5\\
  \midrule
  0.0180	&	Neg.\ stable	&	0	&	18	&	0	&	25	&	0	&	100	&	0	\\
  0.4306	&	Neg.\ stable	&	38	&	29	&	38	&	0	&	45	&	0	&	0	\\
  0.8730	&	Neg.\ stable	&	62	&	40	&	62	&	0	&	47	&	0	&	0	\\
  0.0006	&	Pos.\ stable	&	0	&	13	&	0	&	19	&	8	&	0	&	0	\\
  0.0018	&	Unstable	&	0	&	0	&	0	&	25	&	0	&	0	&	32	\\
  0.0033	&	Unstable	&	0	&	0	&	0	&	35	&	0	&	0	&	37	\\
  0.2294	&	Unstable	&	0	&	0	&	0	&	0	&	0	&	0	&	31	\\
  \bottomrule
\end{tabular} \\
\vskip10pt
    \includegraphics[width=\textwidth]{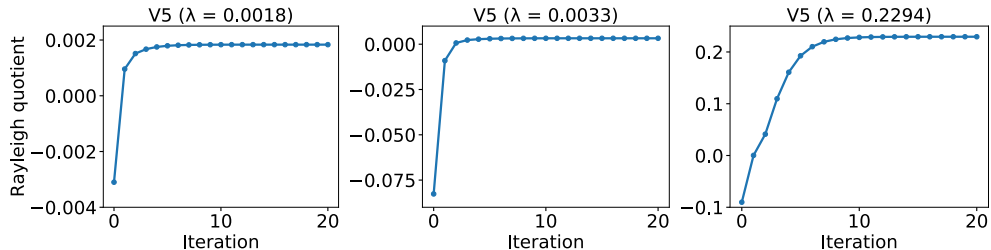}
    \vspace{-6mm}    
    \caption{\textbf{(Top)} The 7 eigenvalues of the test tensor from
      Kolda and Mayo~\cite[Example 3.6]{Kolda-2011-sshopm} and the number of 
  random trials (out of 100) that converge to the eigenvalue for
  (i) the symmetric higher-order power method; S-HOPM~\cite{DeLathauwer-2000-best,Regalia-2000-hopm,Kofidis-2002-best},
  (ii) the shifted symmetric higher-order power method; SS-HOPM~\cite{Kolda-2011-sshopm}), and 
  (iii) 5 variations of our dynamical systems approach.  
  V1 selects the largest magnitude eigenvalue,
  V2 selects the smallest magnitude eigenvalue, 
  V3 selects the largest algebraic eigenvalue, 
  V4 selects the smallest algebraic eigenvalue, and
  V5 selects the second smallest algebraic eigenvalue.
  Results for (i) and (ii) are from Kolda and Mayo~\cite{Kolda-2011-sshopm}.
  Our algorithm is the only one
  that is able to compute all of the eigenvalues, including
  those which are ``unstable,'' the eigenvectors to which SS-HOPM and S-HOPM cannot
  converge~\cite{Kolda-2011-sshopm}.
  \textbf{(Bottom)} Convergence plots for the three unstable eigenvalues from variation 5
  of our algorithm in terms of the Rayleigh quotient \usebox{\cmTrq},
  where $\vx_k$ is the $k$th iterate.
  }
\label{fig:dynsys-SSHOPM}
\end{figure}

\subsection{Example 3.6 from Kolda and Mayo~\cite{Kolda-2011-sshopm}}\label{sec:ex_KM}
Our first test case is a $3 \times 3 \times 3$
symmetric tensor from Kolda and Mayo~\cite[Example 3.6]{Kolda-2011-sshopm}:
\[
\cT_{:,:,1} = \bmat{
 -0.1281 &  \phantomneg0.0516  & -0.0954 \\
  \phantomneg0.0516 & -0.1958  & -0.179  \\
 -0.0954 & -0.179   & -0.2676
},\quad
\cT_{:,:,2} = \bmat{
  \phantomneg0.0516  & -0.1958  & -0.179  \\
 -0.1958   & \phantomneg0.3251  & \phantomneg0.2513 \\
 -0.179    & \phantomneg0.2513   & \phantomneg0.1773
}
\]
\[
\cT_{:,:,3} = \bmat{
 -0.0954 & -0.179   & -0.2676 \\
 -0.179   & \phantomneg0.2513  & \phantomneg0.1773 \\
 -0.2676  & \phantomneg0.1773 &  \phantomneg0.0338
}.
\]
The tensor has 7 eigenvalues, which Kolda and Mayo classify
as ``positive stable'', ``negative stable'', or ``unstable'' 
(see \cref{fig:dynsys-SSHOPM}, top), corresponding
to positive definiteness, negative definiteness, or indefiniteness
of the projected Hessian of the Lagrangian of their optimization function~\cite{Kolda-2011-sshopm}.
(Since the tensor has an odd number of modes, we only consider eigenvalues up to sign.)
Kolda and Mayo showed that their shifted symmetric higher-order power method
(SS-HOPM), a generalization of the symmetric higher-order power method 
(S-HOPM)~\cite{DeLathauwer-2000-best,Regalia-2000-hopm,Kofidis-2002-best},
only converges to eigenvectors of the positive or negative stable eigenvalues.
An adaptive version of SS-HOPM has the same shortcoming~\cite{Kolda-2014-GEAP}.
A recently proposed Newton iteration can converge to eigenpairs where the projected
Hessian has eigenvalues bounded away from $0$~\cite{Jaffe-2018-Newton}.

Of the 7 eigenpairs for the above tensor, 3 are unstable. 
Our dynamical systems approach can compute all 7 eigenpairs,
using 5 variations of the dynamical system:
\begin{enumerate}
\item $\Lambda$ maps $\mM$ to the eigenvector with largest magnitude eigenvalue;
\item $\Lambda$ maps $\mM$ to the eigenvector with smallest magnitude eigenvalue;
\item $\Lambda$ maps $\mM$ to the eigenvector with largest algebraic eigenvalue;
\item $\Lambda$ maps $\mM$ to the eigenvector with smallest algebraic eigenvalue; and
\item $\Lambda$ maps $\mM$ to the eigenvector with second smallest algebraic eigenvalue.
\end{enumerate}
We used the forward Euler method with step size set to 0.5 in order to compute the eigenvalues.
Empirically, convergence is fast, requiring fewer than 10 iterators
(\cref{fig:dynsys-SSHOPM}, bottom row).
One can also also compute these eigenvectors with 
semidefinite programming~\cite{Cui-2014-real},
although the scalability of such methods is limited (see \cref{sec:scalability}).
We next provide numerical results from a tensor in this literature.

\begin{figure}[tb]
\begin{minipage}[c]{0.7\textwidth}
\setlength{\tabcolsep}{6pt}
\centering
\begin{tabular}{@{}lcccccccc@{}}
  \toprule
  $\lambda$ & SDP & V1 & V2 & V3 & V4 & V5 \\
  \midrule 
  9.9779 & $\checkmark$  & 94  & 0     & 0     & 100 & 0 \\
  4.2876 & $\checkmark$  & 6    & 0     & 100 & 0     & 0 \\
  0.0000 & $\checkmark$  & 0    & 100 & 0     & 0     & 100 \\
  \bottomrule
\end{tabular} 
\vskip10pt
   \includegraphics[width=\linewidth]{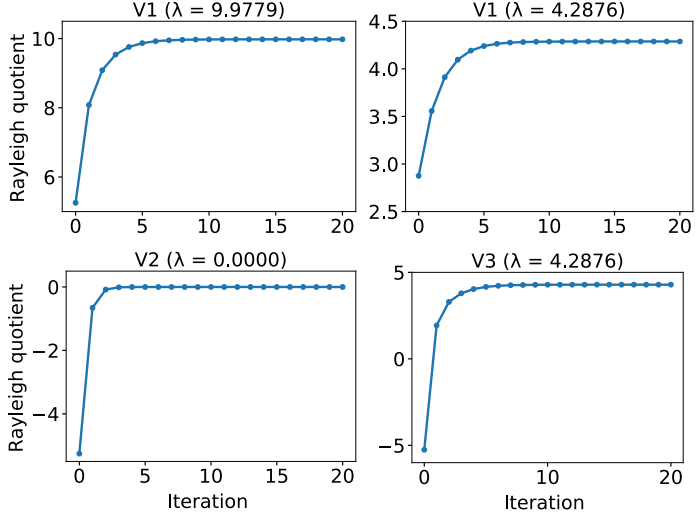}
\end{minipage}
\hfill
\begin{minipage}[c]{0.28\textwidth}
\vspace{-2mm}
\caption{\textbf{(Top)} The 3 eigenvalues of the test tensor from
  from Cui et al.~\cite[Example 4.11]{Cui-2014-real} and the number of 
  random trials (out of 100) that converge to the eigenvalue for
  5 variations of our dynamical systems approach.
  Variation 1 selects the largest magnitude eigenvalue,
  variation 2 selects the smallest magnitude eigenvalue, variation 3 selects the
  largest algebraic eigenvalue, variation 4 selects the smallest algebraic
  eigenvalue, variation 5 selects the 2nd smallest algebraic eigenvalue.
  Our algorithm is able to compute all of the eigenvalues,
  which the SDP approach is guaranteed to compute.
  \textbf{(Bottom)} Convergence plots for the three eigenvalues from different
  variations of our algorithm in terms of the Rayleigh quotient \usebox{\cmTrq},
  where $\vx_k$ is the $k$th iterate.
  }
\label{fig:dynsys-SDP}
\end{minipage}
\end{figure}

\subsection{Example 4.11 from Cui et al.~\cite{Cui-2014-real}}\label{sec:ex_CDN}
Our second test case is a $5 \times 5 \times 5$
symmetric tensor from Cui et al.~\cite[Example 4.11]{Cui-2014-real}:
\begin{align}\label{eq:ex411}
\cT_{i,j,k} = \frac{(-1)^i}{i} + \frac{(-1)^j}{j} + \frac{(-1)^k}{k},\;\; 1 \le i, j, k \le 5.
\end{align}
The tensor has 3 eigenvalues (again, the tensor has an odd number of modes, 
so the eigenvalues are only defined up to sign).
We use the same 5 variations of our algorithm to compute the eigenpairs (\cref{fig:dynsys-SDP}).
Again, we are able to compute all of the eigenvalues of the tensor,
and convergence is rapid.

\subsection{Scalability experiments}\label{sec:scalability}
Now we compare the performance of our algorithm to both SS-HOPM
(as implemented in the Tensor Toolbox for MATLAB\footnote{\url{https://www.tensortoolbox.org/}}~\cite{TTB_Dense,TTB_Software}%
)
and the SDP method 
(also implemented in MATLAB%
\footnote{\url{http://www.math.ucsd.edu/~njw/CODES/reigsymtensor/areigstsrweb.html}}). 
Our implementation is written in Julia and is also publicly available.%
\footnote{\url{https://github.com/arbenson/TZE-dynsys}}

\begin{figure}[tb]
\includegraphics[width=\textwidth]{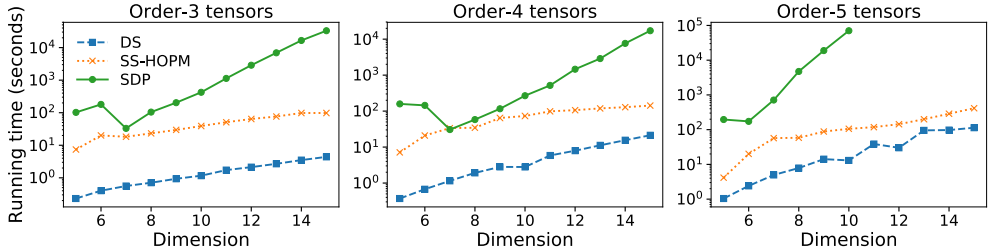}
\vspace{-7mm}
\caption{Running time for 
SS-HOPM~\cite{Kolda-2011-sshopm}, 
the SDP method~\cite{Cui-2014-real},
and our dynamical system method on the test tensor in \cref{eq:test_tensor}.
While the SDP method does not scale, 
the approach has guarantees on the eigenvalues it can compute. 
Our dynamical systems approach uses the map for the eigenvectors of the
$k$th largest algebraic and magnitude (for $k = 1, \ldots, n$), 
along with the forward Euler numerical integration scheme 
with step size 0.5.
We terminated the SDP method if it did not complete within 24 hours,
so only results for dimension up to 10 appear in the plot on the far right.}
\label{fig:scalability}
\end{figure}

The following order-$m$, $n$-dimensional tensor (which is a generalization of \cref{eq:ex411}) serves as the test case for our experiments:
\begin{align}\label{eq:test_tensor}
\cT_{i_1,\ldots,i_m} = \sum_{r=1}^{m} \frac{(-1)^r}{r},\; 1 \le i_1,\ldots,i_m \le n.
\end{align}
With SS-HOPM, we use a tolerance of $10^{-6}$, a shift of 1, and $100n$ random initializations.
With the SDP method, we use the default parameter settings.
With our dynamical systems method, we use a stopping tolerance of $10^{-6}$,
the forward Euler integration scheme with step size = 0.5, and
maps $\Lambda$ corresponding to $k$th largest
algebraic and magnitude eigenvalue, $k = 1, \ldots,n$, each with 50 trials
of random initial starting points. 
With this setup, SS-HOPM and our approach use the same number of randomly initialized trials.
We performed all experiments on a 3.1 GHz Intel Core i7 MacBook Pro with 16 GB of RAM.

\Cref{fig:scalability} shows the running times of the algorithms
for $m = 3,4,5$ and $n = 5,6,\ldots,15$.
The main takeaway is that the SDP method is much slower than the
other two methods---this is the price we pay for being able to compute
all of the real eigenvalues and dealing with NP-hardness.
Our dynamical systems approach is faster than SS-HOPM,
which is somewhat surprising since we require an eigendecomposition
of an $n \times n$ matrix at each iteration.
However, the performance difference is a result of rapid convergence,
as observed in \cref{fig:dynsys-SSHOPM,fig:dynsys-SDP}.
Finally, although the tensors here are relatively small, our method
has been used in recent work to compute eigenvectors of tensors
of order 3, 4, and 5 with dimensions in the tens of thousands~\cite{Benson-2018-hypergraphs}.

\begin{figure}[tb]
\centering
\includegraphics[width=0.495\textwidth]{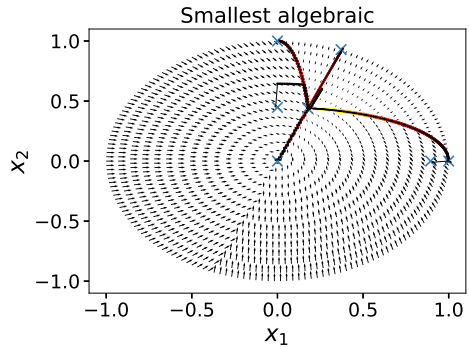}
\vspace{-5mm}
\caption{%
Vector field of dynamical systems from \cref{eqn:dynsys}, along with trajectories of the dynamical system, 
akin to \cref{fig:dynamics_normalized} (right).
However, here we normalize the vector to have unit 2-norm after each step of numerical integration of the dynamical system.
This helps converge to the attractor, where the smallest algebraic eigenvalue map is undefined. A better understanding of
why this helps is an avenue for future research.
}
\label{fig:dynamics_normalized}
\end{figure}

\section{Stochastics as a guide}

Scalable methods for computing tensor eigenvectors remain a challenge. Our new
framework for computing $Z$-eigenvectors offers insights through three
observations. First, a tensor $Z$-eigenvector is a matrix eigenvector of
\emph{some} matrix, where the matrix is obtained by applying the tensor collapse
operator with the $Z$-eigenvector itself. Second, for a certain class of tensors
where eigenvectors have a stochastic interpretation, the dynamical system in
\cref{eqn:dynsys} is the one that governs the long-term dynamics of the
stochastic process. Third, the same type of dynamical system seems to work for
more general tensors. This framework can compute tensor eigenvectors that other
scalable methods, such as the shifted higher-order power method (SS-HOPM),
cannot.

The dynamical system framework is a flexible setup to create solvers for tensor
eigenvector problems, and dynamical systems have also been used in matrix
eigenvector problems~\cite{Chu-1984-Toda,Golub-2006-continuous}. The difference
between SS-HOPM and our proposed framework, for instance, is essentially that
SS-HOPM takes a single step of the power method on the matrix $\cmT[\vx]$,
whereas we converge to an eigenvector of $\cmT[\vx]$. There is a rich space to
interpolate between these positions. Straightforward ideas include low-degree
polynomial filters that target specific eigenvectors.

Indeed, one major challenge is knowing what map $\Lambda$ to choose---different
choices lead to different eigenvectors and there is no immediate relationship
between them for general tensors. A second class of open questions relates to
convergence theory. At the moment, we have demonstrated that there are both
convergent and non-convergent cases (\cref{fig:dynamics}). We can alleviate
this problem by normalizing the iterates of the integration scheme to have unit 2-norm
after each iteration (\cref{fig:dynamics_normalized}); however, we do not have a good
theory for why this works. Finally, our method is not immediately
applicable to $H$-eigenvectors because \cref{obs:main} no longer holds. Adapting
our methodology to this class of eigenvectors is an area for future research.

Our framework came from relating tensor eigenvectors to stochastic processes.
This is quite different from the core ideas in the tensor literature, which are
firmly rooted in algebraic generalizations. We hope that these results encourage
further development of the relationships between stochastics and tensor
problems.

\section*{Acknowledgments}
We thank Brad Nelson for providing valuable feedback.
ARB is supported in part by NSF award DMS-1830274 and ARO award W911NF-19-1-0057.
DFG is supported by NSF award CCF-1149756, IIS-1422918, IIS-1546488, the
NSF Center for Science of Information STC, CCF-0939370, DARPA SIMPLEX, NASA,
DOE DE-SC0014543, and the Sloan Foundation.

\bibliographystyle{abbrv}
\bibliography{refs}
\end{document}

%% file: preamble.tex
\usepackage{amsfonts}
\usepackage{graphicx}
\usepackage{epstopdf}

\usepackage{amssymb} 
\usepackage{amsmath} 
\usepackage{amsbsy}
\usepackage{tabularx}
\usepackage{tikz,nicefrac,amsmath,pifont,tkz-graph,makecell,tkz-graph}
\usetikzlibrary{arrows,backgrounds,patterns,matrix,shapes,fit,shadows,plotmarks}
\usepackage{booktabs}
\usepackage{microtype}

\newcommand{\phantomneg}{\phantom{-}}

\usepackage[notheorems]{dgleich-math}
\usepackage{dgleich-tensor}

\usepackage{listings}
\usepackage[T1]{fontenc}  
\lstdefinelanguage{Julia}%
{morekeywords={abstract,break,case,catch,const,continue,do,else,elseif,end,export,false,for,function,immutable,import,importall,if,in,macro,module,otherwise,quote,return,switch,true,try,type,typealias,using,while},%
	sensitive=true,alsoother={$},morecomment=[l]\#,morecomment=[n]{\#=}{=\#},morestring=[s]{"}{"},morestring=[m]{'}{'},}[keywords,comments,strings]%
\newsiamremark{remark}{Remark}
\newsiamremark{observation}{Observation}
\newsiamremark{hypothesis}{Hypothesis}
\crefname{hypothesis}{Hypothesis}{Hypotheses}
\newsiamthm{claim}{Claim}

\definecolor{mylinkcolor}{RGB}{0,0,140}
\hypersetup{colorlinks,allcolors=mylinkcolor,citecolor=mylinkcolor}